\documentclass[10pt, english]{amsart}

\usepackage{amsmath,amssymb,enumerate}

\usepackage[T1]{fontenc}
\usepackage[all]{xy}

\usepackage{babel}
\usepackage{amstext}
\usepackage{amsmath}
\usepackage{amsfonts}
\usepackage{latexsym}
\usepackage{ifthen}

\usepackage{xypic}
\xyoption{all}
\newcommand{\E}{{\mathcal E}}

\newcommand{\rk}{{\rm rk}}
\newcommand{\codim}{{\rm codim}}

\newtheorem{lemma1}{}[section]

\newenvironment{lemma}{\begin{lemma1}{\bf Lemma.}}{\end{lemma1}}

\newenvironment{theorem}{\begin{lemma1}{\bf Theorem.}}{\end{lemma1}}
\newenvironment{proposition}{\begin{lemma1}{\bf Proposition.}}{\end{lemma1}}
\newenvironment{corollary}{\begin{lemma1}{\bf Corollary.}}{\end{lemma1}}
\newenvironment{remark}{\begin{lemma1}{\bf Remark.}\rm}{\end{lemma1}}
\newenvironment{remarks}{\begin{lemma1}{\bf Remarks.}\rm}{\end{lemma1}}
\newenvironment{definition}{\begin{lemma1}{\bf Definition.}}{\end{lemma1}}
\newenvironment{notation}{\begin{lemma1}{\bf Notation.}}{\end{lemma1}}

\newenvironment{remark*}{{\bf Remark.}}{}
\newenvironment{example*}{{\bf Example.}}{}
\newenvironment{assumption*}{{\bf Assumption.}}{}

\newcommand{\R}{\ensuremath{\mathbb{R}}}
\newcommand{\Q}{\ensuremath{\mathbb{Q}}}

\newcommand{\N}{\ensuremath{\mathbb{N}}}
\newcommand{\PP}{\ensuremath{\mathbb{P}}}

\usepackage{eurosym}

\newcommand{\holom}[3]{\ensuremath{#1:#2  \rightarrow #3}}
\newcommand{\fibre}[2]{\ensuremath{#1^{-1} (#2)}}

\makeatletter
\ifnum\@ptsize=0  \fi \ifnum\@ptsize=2  \fi 

\newcommand\sE{{\mathcal E}}

\newcommand\sF{{\mathcal F}}
\newcommand\sG{{\mathcal G}}
\newcommand\sH{{\mathcal H}}
\newcommand\sI{{\mathcal I}}

\newcommand\sO{{\mathcal O}}

\DeclareMathOperator*{\nons}{nons}

\title{Algebraic integrability of foliations with numerically trivial canonical bundle} 
\date{April 19, 2018}

\author{Andreas H\"oring}
\author{Thomas Peternell}

\address{Andreas H\"oring, Universit\'e C\^ote d'Azur, CNRS, LJAD, France}
\email{Andreas.Hoering@unice.fr}

\address{Thomas Peternell, Mathematisches Institut, Universit\"at Bayreuth, 95440 Bayreuth, 
Germany}
\email{thomas.peternell@uni-bayreuth.de}

\subjclass[2010]{14J32, 37F75, 14E30}
\keywords{MMP, minimal models, algebraic integrability, positivity of vector bundles}

\begin{document}

\begin{abstract} 
Given a reflexive sheaf on a mildly singular projective variety, we prove 
a flatness criterion under certain stability conditions. 
This implies the algebraicity of leaves for sufficiently stable foliations with numerically trivial canonical bundle
such that the second Chern class does not vanish.
Combined with the recent work of Druel and Greb-Guenancia-Kebekus this establishes
the Beauville-Bogomolov decomposition for minimal models with trivial canonical class. 
\end{abstract}

\newpage

\maketitle

\section{introduction} 

\subsection{Main result}
Let $X$ be a normal complex projective variety that is smooth in codimension two, and let $\sE$ be a reflexive
sheaf on $X$. 
 If $\sE$ is slope-stable with respect to some ample divisor $H$  of slope $\mu_H(\sE)=0$,
then a famous result of Mehta-Ramanathan \cite{MR84} says that the restriction $\sE_C$ to a general complete
intersection $C$ of {\em sufficiently ample} divisors is stable and nef.
On the other hand the variety $X$ contains many dominating families of irreducible curves to which 
Mehta-Ramanathan does not apply; therefore one expects that, apart from very special situations, $\sE_C$ will not be
nef for many curves $C$. 
If $\sE$ is locally free, denote by $\holom{\pi}{\PP(\sE)}{X}$ the projectivisation of $\sE$
and by $\zeta:=c_1(\sO_{\PP(\sE)}(1))$ the tautological class on $\PP(\sE)$. The nefness of $\sE_C$ then translates into the nefness of the restriction of $\zeta$ to $\PP(\sE_C)$.
Thus, the stability of $\sE$ implies some positivity of the tautological class $\zeta$. 
On the other hand, the non-nefness of $\sE_C$ on many curves can be rephrased by saying that 
the tautological class $\zeta$ should not be pseudoeffective. 
The first main result of this paper confirms this expected picture under some additional 
stability condition.

\begin{theorem} \label{theoremmain}
Let $X$ be a normal $\Q$-factorial projective variety of dimension $n$ with at most klt singularities.
Suppose that $X$ is smooth in codimension two. 
Let $H$ be an ample line bundle on $X$, and let 
$\sE$ be a reflexive sheaf of rank $r$ on $X$ such that $c_1(\sE) \cdot H^{n-1}=0$. 
Suppose that
\begin{enumerate}[(a)]
\item the reflexive symmetric powers $S^{[l]} \sE$ are $H$-stable for every $l \in \N$, and
\item the algebraic holonomy group of $\sE$ (cf. Definition \ref{def:holonomy}) is connected.
\end{enumerate}
Suppose further that $\sE$ is pseudoeffective (cf. Definition \ref{definitionpseff}). Then 
$$
c_2(\sE) \cdot H^{n-2}=0.
$$
Moreover, there exists a finite Galois cover $\holom{\nu}{\tilde X}{X}$, \'etale in codimension one, 
such that the reflexive pull-back $\nu^{[*]} \sE$ is a numerically flat vector bundle; in particular, $c_2(\nu^{[*]}\sE ) = 0$.
If $X$ is smooth, then $\sE$ itself is a numerically flat vector bundle.
\end{theorem}

Nakayama \cite[Thm.B]{Nak99b} and Druel \cite[Thm.6.1]{Dru17} obtained similar results for vector bundles
of small rank. The recent progress on algebraic integrability of foliations by Campana-P\v aun \cite{CP15} and Druel \cite{Dru17}
yields an immediate application:

\begin{corollary} \label{cor:leaves} 
Let $X$ be a simply connected projective manifold, and let $\sF \subset T_X$ be an integrable reflexive subsheaf. Suppose there exists
an ample line bundle $H$ on $X$ such that $S^{[l]} \sF$ is $H-$stable for all $l \in \N$.
If $c_1(\sF)=0$ and $c_2(\sF) \neq 0$, then $\sF$ has algebraic leaves. 
\end{corollary}

By \cite[Prop.5]{BK08} the stability of all the $S^{[l]} \sF$ is equivalent to assuming that $\sF$ is stable
and the algebraic holonomy group is $\mbox{SL}(F_x)$ or $\mbox{Sp}(F_x)$.
It seems possible that the stability of $\sF$ is enough to imply the algebraicity of leaves (cf. 
also  \cite{Tou08, LPT11, LPT13, PT13} for classification results of foliations with $c_1(\sF)=0$).

The proof of Theorem \ref{theoremmain} is surprisingly simple. Druel's proof \cite[Thm.6.1]{Dru17} uses the
stability of $\sE$ to describe the components of the restricted base locus $B_-(\zeta)$ (see Section \ref{subsectionrestricted}) that are divisors or
generically finite over the base. Our key observation is that
the systematic use of the symmetric powers $S^{[l]} \sE$ allows to control irreducible components of $B_-(\zeta)$
of any codimension. An intersection computation essentially reduces Theorem \ref{theoremmain} to the following:

\begin{proposition} \label{propositioncurve}
Let $C$ be a smooth projective curve. Let $\sE$ be a vector bundle
on $C$ such that $c_1(\sE)=0$. Suppose that the symmetric powers $S^l \sE$ are stable for every $l \in \N$.
Then, given any integer $ 0 \leq d  \leq \dim \PP(\sE),$ the intersection number 
$$
\zeta^d \cdot Z>0
$$
for every subvariety $Z \subset \PP(\sE)$ of dimension $d$.
\end{proposition}

A well-known result of Mumford \cite[Ex.10.6]{Har70} says that if $\sE$ is a locally free sheaf of rank two on a curve $C$
such that $c_1(\sE)=0$ and all the symmetric powers $S^l \sE$ are stable, then the tautological class
$\zeta$ has positive intersection with every curve $Z \subset \PP(\sE)$. Our proposition generalises this property to vector bundles of arbitrary rank.

\subsection{Minimal models with trivial canonical class}

The main motivation for our study of stable sheaves with numerically trivial determinant is to extend the Beauville-Bogomolov
decomposition \cite{Bea83} to singular spaces. Following \cite{GKP16}, let us explain the notions of 
singular Calabi-Yau and singular irreducible symplectic varieties.

\begin{definition} \label{definitionCYsymplectic}
Let $X$ be a normal projective variety of dimension $n \geq 2$ with at most canonical singularities
such that $\omega_X \simeq \sO_X$. 
\begin{itemize} 
\item $X$ is a Calabi-Yau variety if $H^0(Y,\Omega^{[q]}_Y) = 0$ for all integers $1 \leq q  \leq n-1$ and all covers $Y \to X$, \'etale in codimension $1$;
\item $X$ is irreducible symplectic if there exists a holomorphic $2$-form $\sigma \in H^0(X,\Omega^{[2]}_X) $ such that for 
 all covers $\gamma:Y \to X$, \'etale in codimension $1$, the exterior algebra of holomorphic reflexive forms is generated by the reflexive pull-back $\gamma^{[*]}(\sigma).$ 
\end{itemize} 
\end{definition} 

The Beauville-Bogomolov decomposition for a Ricci flat compact K\"ahler manifold $X$ states that after \'etale cover
$X$ is a product of a torus, Calabi-Yau and irreducible symplectic manifolds. In the last years there has been an intensive
effort \cite{GKP16c, Dru17, GGK17, DG17} to generalise this statement to minimal models.
Theorem \ref{theoremmain} allows to complete this challenge:

\begin{theorem} \label{theoremdecomp} 
Let $X$ be a normal projective variety with at most klt singularities such that $c_1(K_X)=0$.

Then there exists a projective variety $\tilde X$ with only canonical 
singularities, a quasi-\'etale map $f: \tilde X \to X$ and
a decomposition 
$$ 
\tilde X \simeq A \times \prod_{j \in J}  Y_j \times \prod_{k \in K}  Z_k
$$
into normal projective varieties with trivial canonical bundles, such that 
\begin{itemize} 
\item $A$ is an abelian variety; 
\item the $Y_j$ are (singular) Calabi-Yau varieties;
\item the $Z_k$ are (singular) irreducible symplectic varieties. 
\end{itemize} 
\end{theorem} 

Although this significantly improves results from earlier papers, one should note that Theorem \ref{theoremdecomp}
is based to equal parts on a tripod consisting of Druel's algebraic integrability theorem \cite[Thm.1.4]{Dru17}, 
the holonomy decomposition of Greb-Guenancia-Kebekus \cite[Thm.B and Prop.D]{GGK17} and our Theorem \ref{theoremmain}.
For the proof we simply follow the arguments of \cite[Thm.1.6]{Dru17}.
 
Another consequence of Theorem \ref{theoremmain} is 

\begin{theorem} \label{theorempseff}
Let $X$ be a normal projective variety with at most canonical singularities. Suppose that $X$ is smooth in codimension two
and $c_1(K_X)=0$. Assume that the tangent sheaf  $T_X$ is strongly stable in the sense of \cite[Defn.7.2]{GKP16c}.

Then both the reflexive cotangent sheaf $\Omega_X^{[1]}$ and the tangent sheaf $T_X$ are not pseudoeffective.
In particular if $X$ is Calabi-Yau or irreducible symplectic manifold (in the sense of Definition \ref{definitionCYsymplectic}), 
then $\Omega_X^{[1]}$ and $T_X$ are not pseudoeffective (cf. Definition \ref{definitionpseff}).
\end{theorem}

This result was proven for  surfaces in \cite[Thm.IV.4.15]{Nak04}  \cite[Thm.7.8]{BDPP13} and for  threefolds in \cite[Cor.6.5]{Dru17}.

\subsection{Almost nef sheaves}

While Theorem \ref{theoremmain} is sufficiently strong for the proof of the decomposition theorem, it is in general 
not easy to control the stability of all the symmetric powers. We therefore also consider a weaker positivity notion:

\begin{definition} 
Let $X$ be a normal projective variety, and let $\sE$ be a reflexive sheaf on $X$.
We say that $\sE$ is almost nef, if there exist at most countably many
proper subvarieties $S_j \subsetneq X$ such that the following holds: let $C \subset X$ be a curve such that
$\sE \vert_C = \sE \otimes \sO_C$ is not nef, then $C$ is contained in $\cup_{j \in J} S_j$. 
\end{definition}

Using completely different techniques we prove the following

\begin{theorem} \label{theorem:almostnef} 
Let $X$ be a normal $\mathbb Q$-factorial projective variety with at most klt singularities. 
Suppose that $X$ is smooth in codimension $2$. 
Let $\sE$ be a reflexive sheaf on $X$ such that $c_1(\sE)=0$.

If $\sE$ is almost nef, then we have $c_2(\sE)=0$. Moreover
there exists a quasi-\'etale cover $\gamma: \tilde X \to X$ such that  $\gamma^{[*]}(\sE)$ is locally free and numerically flat. 
\end{theorem}

Based on analytic techniques a slightly weaker statement was shown in \cite[Prop.2.11]{CH17}.
In particular we obtain a positive answer to a question asked in \cite{DPS01}, 
without any assumption on the stability:

\begin{theorem} \label{corollaryalmostnef}
Let $X$ be a normal $\mathbb Q$-factorial projective variety with at most klt singularities. 
Suppose that $X$ is smooth in codimension $2$.  
Suppose that $c_1(X)=0$ and $X$ is not dominated by an abelian variety.

Then $T_X$ is not almost nef, i.e. there exists a dominating family of irreducible curves $C_t \subset X_{\nons}$
such that $T_X|_{C_t}$ is not nef.
\end{theorem}

This statement was shown for smooth threefolds in \cite[Thm.7.7]{BDPP13}.
The assumption of Theorem \ref{corollaryalmostnef} is too weak to use the techniques from Theorem \ref{theoremmain}.
Nevertheless we expect that the stronger conclusion of Theorem \ref{theorempseff} also holds for minimal models with
trivial canonical class such that the decomposition does not contain an abelian factor.

{\bf Acknowledgements.} We thank S. Cantat and P. Graf for some very useful references.
This work was partially supported by the Agence Nationale de la Recherche grant project Foliage\footnote{ANR-16-CE40-0008} and by the DFG project "Zur Positivit\"at in der komplexen
Geometrie". 

\section{Notation, basic facts and proof of Proposition \ref{propositioncurve} }
\label{sectionnotation}

We work over the complex numbers, for general definitions we refer to \cite{Har77}. 
We use the terminology of \cite{Deb01} and \cite{KM98}  for birational geometry and notions from the minimal model program and \cite{Laz04a} for notions of positivity.
Manifolds and varieties will always be supposed to be irreducible. Given a normal 
variety $X$ we denote by $T_X:= \Omega_X^{*}$ its tangent sheaf. 
The sheaf of reflexive differentials of degree $q \in \{ 1, \ldots, \dim X \}$ is given by 
$$
\Omega^{[q]}_X := (\bigwedge^q\Omega^1_X)^{**}.
$$ 
A finite map $\holom{\gamma}{X'}{X}$ between normal varieties is quasi-\'etale if its ramification divisor
is empty (or equivalently, by purity of branch, $\gamma$ is \'etale over the smooth locus of $X$).

Given a torsion-free sheaf $\sF$ on a normal variety $X$, we denote by $S^{[m]} \sF:= (\mbox{Sym}^m \sF)^{**}$
the $m$-th reflexive symmetric power. Given a morphism $\holom{\nu}{Y}{X}$, we denote by
$\gamma^{[*]}(\sF):=(\gamma^*(\sF))^{**}$ the reflexive pull-back.
The projectivization $\PP(\sF) $ is defined by 
$$ \PP(\sF) =  \mbox{Proj}(\mbox{Sym}^\bullet \sF);$$  
with projection $ p:  \PP(\sF) \to X$.
By $\sO_{\PP(\sF)}(1) = \sO(1) $ we denote the tautological line bundle on $\PP(\sF).$ Hence
$$ H^q(\PP(\sF), \sO(1) \otimes p^*(\sG)) = H^q(X,S^m(\sF) \otimes \sG)) $$
for all $q \geq 0$ and all locally free sheaves $\sG$ on $X$.

In order to simplify the notations we will denote, for a normal subvariety $Y \subset X$ such that
$\sE$ is locally free near $Y$, the restriction $\sE|_Y$
by $\sE_Y$ and by $\zeta_Y$ the restriction of the tautological class to $\PP(\sE_Y)$.

If $\sE$ is a locally free sheaf on a normal projective variety, then $\sE$ is {\it pseudoeffective } if the line bundle
$\sO_{\PP(\sE)}(1)$ is pseudoeffective. This is equivalent to saying that, fixing any ample line bundle $H$ on $X$,  
for any real number $c > 0$ there exist positive integers $i$ and $j$ with $i > cj $ 
such that
$$
H^0(X, S^i \sE \otimes H^{\otimes j}) \ne  0.
$$
This vanishing will be crucial in our argumentation.

If $\sE$ is merely reflexive, then 
in general the projectivisation 
$\PP(\sE) $ is a very singular space and the push-forward of multiples
of the tautological class are not isomorphic to the reflexive symmetric powers $S^{[m]} \sE$. Because of this subtlety,
we choose to define pseudoeffective reflexive sheaves using the cohomological characterization above:

\begin{definition} \label{definitionpseff}
Let $X$ be a normal variety, and let $\sE$ be a reflexive sheaf on $X$. We say that $\sE$ is pseudoeffective
if for all $c>0$ there exist a number $j \in \N$ and $i \in \N$ such that $i>cj$ and
$$
H^0(X, S^{[i]} \sE \otimes H^{\otimes j}) \neq 0.
$$ 
\end{definition}

In order relate this to $\PP(\sE) $ and the tautological sheaf $\sO_{\PP(\sE)}(1)$, 
we use a construction due to Nakayama \cite[V,3.23]{Nak04}:

\begin{definition} \label{definitiontautological}
Let $X$ be a normal variety, and let $\sE$ be a reflexive sheaf on $X$. 
\begin{itemize} 
\item 
Denote by $$ \nu: \PP'(\sE) \to \PP(\sE)$$ 
the normalization of the unique component of $\PP(\sE)$ 
that dominates $X$. 
\item Set  $\sO_{\PP'(\sE)}(1) := \nu^*(\sO_{\PP(\sF)}(1))$. 
\item 
Let $X_0 \subset X$ be the locus where $X$ is smooth and  $\sE$ is locally free, and let 
$$\holom{r}{\PP_0(\sE)}{\PP'(\sE)}$$ 
be a birational morphism from a manifold $\PP_0(\sE)$ such
that the complement of $\fibre{(p \circ \nu \circ r)}{X_0} \subset \PP_0(\sE) $ is a divisor $D$.
\item Set $\pi := p \circ \nu \circ r$ and $\sO_{\PP_0(\sE)}(1) := r^*(\sO_{\PP'(\sE)}(1)).$ 
\item  By \cite[III.5.10.3]{Nak04} there exists an effective
divisor $\Lambda$ supported on $D$ such that
$$
\pi_* (\sO_{\PP_0(\sE)}(1) \otimes \sO(\Lambda)) \simeq S^{[m]} \sE \qquad \forall \ m \in \N.
$$
\item  We call $\zeta := c_1(\sO_{\PP_0(\sE)}(1) \otimes \sO(\Lambda))  \in N^1(\PP_0(\sE))$
 a tautological class of $\sE$. 
 
\end{itemize}
\end{definition}

Using the defining property of a tautological class, the arguments of \cite[Lemma 2.7]{Dru17} apply literally to show the following: 

\begin{lemma} \label{lemmapseff}
Let $X$ be a normal projective variety, and let $H$ be an ample line bundle on $X$.
Let $\sE$ be a reflexive sheaf on $X$, and let $\zeta $ be a tautological class on $\holom{\pi}{\PP_0(\sE)}{X}$.
Then $\zeta $ is pseudoeffective if and only if $\sE $ is pseudoeffective. 
\end{lemma}

\begin{remarks}

- Lemma \ref{lemmapseff} shows in particular that the existence of a pseudoeffective class $\zeta$  does not depend 
on the choice of the birational model $\PP_0(\sE) \rightarrow \PP'(\sE) \rightarrow X$,
nor on the effective divisor $\Lambda$. 

- If the tautological class $\sO_{\PP'(\sE)}(1)$ on the normalisation $\PP'(\sE)$ is pseudoeffective,
any tautological divisor $\zeta$ is pseudoeffective. 
%\item The definition \ref{definitionpseff} is not suitable for torsion-free sheaves that are not reflexive. In fact if we applied the construction 
%verbatim to an ideal sheaf $\sI$ defining a subvariety $Z \subset X$ of codimension at least $2$, then it would be pseudoeffective.  

\end{remarks}

\begin{definition} \cite{DPS94}
Let $X$ be a normal projective variety, and let $\sE$ be a locally free sheaf on $X$. We say that $\sE$ is numerically flat
if both $\sE$ and $\sE^*$ are nef. This is equivalent to assuming that both $\sE$ and $\det \sE^*$ are nef. 
\end{definition}  

\begin{remarks} \label{remarksflat}

- By \cite{DPS94} we know that if $\sE$ is numerically flat, then all the Chern classes vanish and $\sE$ is semi-stable for any ample polarization. 

- Let $\sE$ be a locally free sheaf on $X$ that is flat, i.e. $\sE$ is given by a linear representation of $\pi_1(X)$.
If $c_1(\det \sE)=0$, then $\sE$ is nef (e.g. \cite[Thm.1.1]{JR13}), hence numerically flat. 

- Strictly speaking, both statements are established so far only when $X$ is smooth. However, it is not difficult to derive the assertions in the normal case by
passing to a desingularisation; see \cite{GKP16} for the techniques. 
\end{remarks}

\begin{definition} Let $X$ be a normal projective variety of dimension $n$ that is smooth in codimension two,
and let $H$ be an ample divisor on $X.$
Let $\sE$ be a reflexive free coherent sheaf on $X$. 

For $m \gg 0$ let $D_j$ be general divisors in $\vert mH \vert $ and set 
$$ 
C = D_1 \cap \ldots \cap H_{n-1}; \  S =  D_1 \cap \ldots \cap H_{n-2}.
$$ 
We define 
$$ 
c_1(\sE) \cdot H^{n-1} := \frac {1}{m^{n-1}}  c_1(\sE \vert C).
$$ 
We furthermore set 
$$ 
c_2(\sE) \cdot H^{n-2} :=  \frac {1}{m^{n-2}}  c_2(\sE \vert S),
$$ 
observing that $\sE_S$ is locally free. 
\end{definition} 

The definitions above do not depend on the choice of $m$ and the divisors $D_j.$ If $X$ is $\mathbb Q-$factorial, then $\det \sE$ is $\mathbb Q-$Cartier and $c_1(\sE)$ itself is defined by $\frac{1}{m} c_1(\det \sE^{[m]})$ for sufficiently divisible $m$.
We refer to \cite[Sect.2.8]{Dru17} for a systematic approach.
  
\subsection{Stability and holonomy groups}

In this paper we will use the standard notion of slope-(semi-)stability of torsion-free sheaves $\sF$ with respect to an ample line bundle $H$ as defined in \cite[Part I, Lect.III]{MP97}, and denote by $\mu_H(\sF)$ the slope of $\sF$ with respect to $H$.
Miyaoka has shown the following useful basic fact:

\begin{proposition}  \label{prop:triv}  \cite{Miy87}, \cite[Prop.6.4.11]{Laz04b}
Any semistable vector bundle $\sE$ over a smooth curve with $c_1(\sE)=0$ is nef.
\end{proposition} 

The behaviour of stability under restrictions will play an important role.

\begin{definition} \label{definitionmrgeneral}
Let $X$ be a normal projective variety, and let $H$ be an ample line bundle on $X$.
Let $\sF$ be a torsion-free sheaf on $X$ that is $H$-semistable.  
A MR-general curve $C \subset X$ is a complete intersection
$D_1 \cap \ldots \cap D_{n-1}$ where $D_j \in | m H |$ with $m \in \N$ such that the restriction $\sF_C$ is semistable.
\end{definition}

\begin{remark} \label{remarkmrgeneral}
The abbreviation MR stands of course for Mehta-Ramanathan, alluding to the well-known fact \cite{MR82, Fle84}
that for $m \gg 0$ and a general $D_j \in | m H |$, the restriction $\sF_C$ is indeed semistable.
\end{remark}

For lack of reference we include the following singular version of the restriction theorem of Mehta-Ramanathan 
\cite[Thm.4.3]{MR84}.

\begin{lemma} \label{lemmaMR}
Let $X$ be a normal $\Q$-factorial projective variety of dimension $n$, and $H$ an ample line bundle on $X$. Let $\sE$ be a torsion-free sheaf on $X$ 
that is $H$-stable. Then there exists a $m_0 \in \N$ such that for all $m \geq m_0$ 
and $D_1, \ldots, D_{n-1}$ general elements in $|m H|$, the restriction $\sE_C$ of $\sE$ to
$C := D_1 \cap \ldots \cap D_{n-1}$ is stable.
\end{lemma}

\begin{remark*}
As in Definition \ref{definitionmrgeneral} we will call $C$ a MR-general curve.
\end{remark*}

\begin{proof}
Let $\holom{\mu}{X'}{X}$ be a resolution of singularities. By \cite[Lemma 4.6]{GKP16} the reflexive
pull-back $(\mu^* \sE)^{**}$ is stable with respect to the semiample and big divisor $\mu^* H$.
By \cite[Thm.5.2]{Lan04} for all $k \gg 0$ and $D_1' \in |k \mu^* H|$ general 
the restriction $(\mu^* \sE)^{**}|_{D_1'}$ is stable with respect to $(\mu^* H)_{D_1'}$.
Thus we can find a $m_0 \in \N$ such that for all $m \geq m_0$ and
$D_1', \ldots, D_{n-1}'$ general elements in $|m \mu^* H|$, the restriction 
$(\mu^* \sE)^{**}|_{C'}$
to $C' := D_1' \cap \ldots \cap D_{n-1}'$ is stable. 
Now observe that by the projection formula $$H^0(X',  m \mu^* H) \simeq H^0(X, m H),$$
so the divisors $D_i'$ are strict transforms of general divisors $D_i \in |m H|$. 
Since $X$ is normal, the intersection $C:= D_1 \cap \ldots \cap D_{n-1}$ is in the smooth locus of $X$,
thus the curves $C'$ and $C$ can be identified. Since $\sE$ is torsion-free, hence locally free in codimension one, the sheaves $(\mu^* \sE)^{**}$ and $\sE$ identify in a neighbourhood
of $C=C'$, so $\sE_C =  (\mu^* \sE)^{**}|_{C'}$ is stable.
\end{proof}

\begin{remark} 
(1) We may delete the assumption that $X$ is $\Q$-factorial by replacing the reference \cite[Lemma 4.6]{GKP16}
with the arguments from \cite[Prop.5.2]{Gra16b}. 

(2) More generally, the restriction to $D_1 \cap \ldots \cap D_{k}$ with $1 \leq k \leq n-1$ is $H$-stable.
Indeed if the intermediate restriction is not stable, then the restriction to a MR-general curve is not stable.
\end{remark} 

We recall the notion of the algebraic holonomy, introduced by Balaji and Koll\'ar \cite{BK08}. For convenience, given a reflexive sheaf $\sE$ and $x \in X,$ 
we set
$$ E_x := \sE_x/m_x \sE_x,$$
where $m_x$  is  the maximal ideal at $x.$ 

\begin{definition} \label{def:holonomy} 
Let $X$ be a normal projective variety of dimension $n$, and let $H$ be an ample line bundle on $X.$ 
Let $\sE$ be a reflexive sheaf on $X$ that is $H$-stable with slope $\mu_H(\sE) = 0.$ Fix a smooth point $x \in X$ such that $\sE$ is free near $x$. The algebraic holonomy group of $\sE$ at $x$ 
is the (unique) smallest subgroup 
$$H_x(\sE) \subset GL(E_x) $$
such that the following holds: for any smooth curve $C$ with fixed point $c  \in C$ and any morphism $g: C \to X$ with $g(c) = x$ and $\sE$ locally free near $g(C),$ and such that $g^*(\sE)$ is poly-stable,
the Narasimhan-Seshadri representation $\rho: \pi_1(D,c) \to GL(g^*(E_x))$ has image in $H_x(\sE).$
\end{definition} 

For details and explanations we refer to \cite{BK08}. The following useful lemma is well-known to specialists:

\begin{lemma} \label{lemmaallstable}
Let $C$ be a smooth projective curve of genus at least two, 
and let $\sE$ be a stable vector bundle of rank $r$ on $C$ such that $c_1(\sE)=0$.
Suppose that the algebraic holonomy group $H_x(\sE)$ is  $\mbox{SL}(E_x)$ or $\mbox{Sp}(E_x)$ for a suitable nondegenerate symplectic 
form on $E_x$. Then the symmetric powers $S^m \sE$ are stable for all $m \in \N$.
\end{lemma}

\begin{proof}
By \cite[Sect.12, Cor.1]{NS65} the vector bundle $\sE$ is defined by an irreducible unitary representation 
$\rho: \pi_1(C, x) \rightarrow \mbox{U}(E_x)$, so a symmetric power $S^m \sE$ is defined by
the unitary representation $S^m \rho$ induced on the symmetric power $S^m E_x$. By  \cite[Sect.12, Cor.2]{NS65}
the bundle $S^m \sE$ is stable if and only if the representation $S^m \rho$ is irreducible. Suppose that this is not the case.
Since the image of $\rho$ is dense in the algebraic holonomy group $H_x(\sE)$ \cite[Thm.1(2)]{BK08}, the induced representation 
of $H_x(\sE)$ on $S^m E_x$ is reducible. Yet it is a classical result of the representation theory
of Lie groups \cite{Wey49} that the symmetric representations of $\mbox{SL}(E_x)$ or $\mbox{Sp}(E_x)$ are irreducible.
\end{proof}

\subsection{Subvarieties of projectivised bundles}

We will now prove the key lemma of this paper.

\begin{lemma} \label{lemmacurve}
Let $C$ be a smooth projective curve. Let $\sE$ be a locally free sheaf on $C$ such that $c_1(\sE)=0$. 
Let $Z \subset \PP(\sE)$ be a subvariety of dimension $d$. Denote by $\mathcal I_Z$ the ideal sheaf of $Z$ in $\PP(\sE)$, and let $l \in \N$ be such that $\pi_* (\sI_Z(l))$ has positive rank and  
such that $R^1 \pi_* (\sI_Z(l))=0$. 

Suppose that the locally free sheaf $S^l \sE$ is stable. Then exactly one of the following holds:
\begin{itemize}
\item $\zeta^d \cdot Z>0$; or
\item $Z$ is covered by curves $C'$ such that $\zeta \cdot C' = 0$ and  such that the map $C' \rightarrow C$ is \'etale.
\end{itemize}
\end{lemma}

\begin{remark*} 
In the situation above, $S^l \sE$ is stable with $c_1(\sE)=0$. Thus $S^l \sE$ is nef by Proposition \ref{prop:triv}.
Combined with \cite[Thm.6.2.12]{Laz04b} this implies that the symmetric powers $S^m \sE$ are semistable and nef for all $m \in \N$. 
\end{remark*} 

\begin{proof}[Proof of Lemma \ref{lemmacurve}]
If $Z$ is contained in a fibre of $\pi,$ the statement is trivial, so 
suppose that $\varphi:=\pi|_Z$ is surjective. We denote by $\sO_Z(l)$ the restriction of $\sO_{\PP(\sE)}(l)$ to $Z$.

By our hypothesis, the exact sequence 
$$
0 \rightarrow \sI_Z(l) \rightarrow \sO_{\PP(\sE)}(l) \rightarrow \sO_Z(l) \rightarrow 0
$$
induces an exact sequence
$$
0 \rightarrow \pi_* (\sI_Z(l)) \rightarrow \pi_* (\sO_{\PP(\sE)}(l)) \simeq S^l \sE \rightarrow \varphi_* (\sO_Z(l)) \rightarrow 0
$$
and $\rk S^l \sE> \rk \varphi_* (\sO_Z(l))$. Since $S^l \sE$ is nef by Proposition \ref{prop:triv}, its quotient
$\varphi_* (\sO_Z(l))$ is also nef. If $\varphi_* (\sO_Z(l))$ were not ample, by Hartshorne's theorem \cite[Thm.6.4.15]{Laz04b} there would exist a quotient 
$$
\varphi_* (\sO_Z(l)) \twoheadrightarrow Q
$$ 
such that $c_1(Q)=0$. Yet this quotient
would destabilise $S^l \sE$, so $\varphi_* (\sO_Z(l))$ has to be ample. Since $\sO_Z(l)$
is $\varphi$-globally generated, we have a surjective morphism
$$
\varphi^* \varphi_* (\sO_Z(l)) \twoheadrightarrow \sO_Z(l). 
$$
Since $\varphi^* \varphi_* (\sO_Z(l))$ is semiample, its quotient $\sO_Z(l)$ is a semiample
line bundle. Denote by $\holom{\tau}{Z}{B}$ the morphism with connected fibers defined by some positive multiple
of $\sO_Z(l)$. Then $\tau$ is not birational if and only if 
$$
\zeta^d \cdot Z= c_1(\sO_Z(l))^d = 0.
$$ 
Suppose now that $\zeta^d \cdot Z=0$, so $\tau$ is not birational. The restriction of $\zeta$ to any $\tau$-fibre is numerically trivial, 
so the curves $C'$ contained in the $\tau$-fibres define a covering family such that $\zeta \cdot C'=0$.

Let us now show that for any curve $C' \subset Z$ such that $\zeta \cdot C'=0$ the map $f:= \pi|_{C'}$ 
is \'etale. Note that since $\sO_Z(l)$ is semiample, the restriction $\sO_{C'}(l)$ to $C'$ is a torsion line bundle. 
Thus for a sufficiently divisible $m \gg 0$ the exact sequence
$$
0 \rightarrow \sI_{C'}(m)  \rightarrow \sO_{\PP(\sE)}(m) \rightarrow \sO_{C'}(m) \simeq \sO_{C'} \rightarrow 0
$$
pushes down to an exact sequence
$$
0 \rightarrow \pi_* (\sI_{C'}(m)) \rightarrow S^m \sE \rightarrow f_* (\sO_{C'}(m)) \simeq f_* (\sO_{C'}) \rightarrow 0.
$$
Since $S^m \sE$ is nef, its quotient $f_* (\sO_{C'})$ is a nef vector bundle. 
Now we conclude with Lemma \ref{lemmaetale}.
\end{proof}

\begin{lemma} \label{lemmaetale}
Let $C$ be a smooth projective curve, and let $\holom{f}{C'}{C}$ be a surjective morphism
from a projective, integral curve. If $f_* (\sO_{C'})$ is nef, then $C'$ is smooth and $f$ is \'etale.
\end{lemma}

\begin{proof}
Let $\holom{\nu}{\tilde C}{C'}$ be the normalisation of $C'$. Then we have an exact sequence
$$
0 \rightarrow \sO_{C'} \rightarrow \nu_* (\sO_{\tilde C}) \rightarrow Q \rightarrow 0,
$$
where $Q$ is a torsion sheaf that is non-zero if and only if $\nu$ is an isomorphism.
Pushing down to $C$ we obtain an exact sequence
$$
0 \rightarrow f_* (\sO_{C'}) \rightarrow (f \circ \nu)_* (\sO_{\tilde C}) \rightarrow f_* Q \rightarrow 0.
$$
Since $f_* (\sO_{C'})$ is a nef vector bundle on a smooth curve and the map
$ f_* (\sO_{C'}) \rightarrow (f \circ \nu)_* (\sO_{\tilde C})$ is generically surjective, the vector bundle
$(f \circ \nu)_* (\sO_{\tilde C})$ is nef \cite[Ex.6.4.17]{Laz04b}. Yet by \cite[Cor.1.14]{PS00} this vector bundle is antinef
and has antiample determinant unless $f \circ \nu$ is \'etale. Thus we obtain that  
$(f \circ \nu)_* (\sO_{\tilde C})$ is numerically flat, the map $f \circ \nu$ is \'etale and
$f_* (\sO_{C'}) \simeq (f \circ \nu)_* (\sO_{\tilde C})$. In particular $Q=0$ and $\tilde C \simeq C'$.
\end{proof}

\begin{proof}[Proof of  Proposition \ref{propositioncurve}]
Arguing by contradiction, suppose that there exists a subvariety $Z \subset \PP(\sE)$ of dimension $d$ such that $\zeta^d \cdot Z=0$.
Since $\sO_{\PP(\sE)}(1)$ is $\pi$-ample we know that the assumptions of Lemma \ref{lemmacurve}
are satisfied for $l \gg 0$. By the lemma there exists a curve $C' \subset \PP(\sE)$ such that
$\zeta \cdot C'=0$ and $f: \pi|_{C'} : C' \rightarrow C$ is \'etale.
Choose now $m \in \N$ such that $\pi_* (\sI_{C'}(m))$ has positive rank and  
$R^1 \pi_* (\sI_{C'}(m))=0$. Then we have an exact sequence
$$
0 \rightarrow \pi_* (\sI_{C'}(m)) \rightarrow S^m \sE \rightarrow f_* (\sO_{C'}(m)) \rightarrow 0,
$$
and a Grothendieck-Riemann-Roch computation shows that $c_1(f_* (\sO_{C'}(m)))=0$.
Thus $S^m \sE$ is not stable, a contradiction.
\end{proof}

\section{Reflexive sheaves with pseudoeffective tautological class}

\subsection{Restricted base locus of the tautological class}
\label{subsectionrestricted}

We start by setting up some notations. 

\begin{notation} {\rm 
Let $D$ be a $\Q$-Cartier $\Q$-divisor on a normal projective variety $P$. Then the stable base locus is defined
as 
$$
\mathbb B(D) := \bigcap_{m} \mbox{Bs}(mD),
$$
where the intersection is taken over all $m \in \N$ such that $m D$ is the class of a Cartier divisor, with ${\rm Bs}(mD)$ denoting the base locus of $mD.$ 
The restricted base locus is defined as 
$$
B_-(D) = \bigcup_{\mbox{$A$ ample $\Q$-divisor}} \mathbb B(D+A).
$$ }

\end{notation}

By \cite[Prop.1.19]{ELMNP06} one has
\begin{equation} \label{computeBminus}
B_-(D) = \cup_{n \in \N^*} \mathbb B(D+\frac{1}{n}A)
\end{equation}
where $A$ is an arbitrary ample divisor. Note that if $A$ is very ample then
\begin{equation} \label{inclusionBminus}
\mathbb B(D+\frac{1}{n}A) \subset \mathbb B(D+\frac{1}{n'}A)
\end{equation}
if $n \leq n'$. 

\begin{notation} {\rm  Let $Y$ be a projective manifold, $D$ a pseudoeffective $\mathbb Q-$divisor and $\Gamma $ a prime divisor on $Y$. 
By 
$$ 
\sigma_{\Gamma} (D) = \lim_{\epsilon \to 0^+} \inf
\{
\mbox{mult}_\Gamma (L') \ | \ L' \geq 0 \ \mbox{and} \ [L'] = [L+\epsilon A]
\} 
$$
we define the asymptotic multiplicity or vanishing order of $D$ along $\Gamma$, as defined in \cite[III, Lemma 1.6]{Nak04}; see also \cite[Defn.2.15]{FL17}. 
By definition it  is a numerical invariant of $D$. }

\end{notation} 

We extend this definition to higher codimension, following \cite[III, Defn. 2.2]{Nak04}. 

\begin{notation} {\rm 
Let $P$ be a  projective manifold, and let $D$ be a pseudoeffective $\Q$-Cartier 
divisor on $P$. 
Let $Z \subset P$ be a subvariety, and let 
$$\holom{f}{Y}{P} $$ be the composition of an embedded resolution of $Z$, and the blow-up of the strict transform of $Z$. 
Let $E_Z \subset Y$
be the unique prime divisor mapping onto $Z$. Then we define
$$
\sigma_Z(D) := \sigma_{E_Z}(f^* D).
$$ 
It is easy to check that this definition does not depend on the choice of $f,$ using \cite[III, Lemma 5.15] {Nak04},

Note that by \cite[III, Lemma 1.7(2)]{Nak04} we have
$$
\sigma_Z(D) = \lim_{\epsilon \to 0} \sigma_{E_Z}(f^* (D+\epsilon A))
$$
where $A$ is an arbitrary ample divisor on $P$. Thus $\sigma_Z(D)$ is the asymptotic
vanishing order of $D$ in the generic point of the subvariety $Z$. 

Suppose now that $Z$ is an irreducible component of $B_-(D)$. 
By \cite[Lemma 6.13]{Dru17} we have
\begin{equation} \label{sigmapositive}
\sigma_Z(D) > 0.
\end{equation}

By \cite[III, Lemma 1.7(2)]{Nak04} (applied to $\sigma_{E_Z}(\bullet)$, cf. also 
\cite[p.93, Remark (1)]{Nak04}) the function
$$
\sigma_Z : \overline{\mbox{Eff}}(P) \rightarrow \R
$$
is lower semicontinuous and continuous when restricted to the big cone. }
\end{notation} 

The following technical lemma, an analogue of \cite[Lemma 6.13]{Dru17}, will be very important. 

\begin{lemma} \label{lemmacomputeclass}
Let $X$ be a normal projective variety, and let $H$ be an ample line bundle on $X$. 
Let $\sE$ be a reflexive sheaf of rank $r$ on $X$ that is $H$-semistable with $\mu_H(\sE)=0$. 
Let $$\holom{\pi}{\PP_0(\sE)}{X}$$ be a modification of $\PP(\sE)$ as in Definition \ref{definitiontautological},
and let $\zeta$ be a tautological class on $\PP_0(\sE)$.
Suppose that $\sE$ is pseudoeffective, or equivalently, that $\zeta$ is a pseudoeffective class (Lemma \ref{lemmapseff}).

Fix a positive integer $k \in \N$ and suppose the following:

\begin{itemize}
\item for every irreducible component $W \subset B_-(\zeta)$ of codimension at most $k-1$  the image $\pi(W)$ has codimension at least $2$.
\end{itemize} 

Let $Z \subset B_-(\zeta)$ be an irreducible component of codimension $k$,
and let $C \subset X$ be a very general MR-general smooth curve (with respect to $H$). 

Then there exists $a>0$ such that 
\begin{equation} \label{basicineq} 
\left[
(\zeta|_{\fibre{\pi}{C}})^{k} - a (Z \cap \fibre{\pi}{C}) 
\right] 
\cdot H_1 \cdots H_{r-k} \geq 0
\end{equation} 
for any collection of nef divisors $H_1, \ldots, H_{r-k}$ on $\fibre{\pi}{C}$.
\end{lemma}

\begin{proof}  We will use the shorthand $P = \PP_0(\sE).$ Since  $C \subset X$ is very general and since $B_-(\zeta)$ has 
at most countably many irreducible components we have
$$
\pi(W) \cap C = \emptyset
$$
for every irreducible component $W \subset B_-(\zeta)$ of codimension at most $k-1$,

Moreover, the sheaf $\sE$ is locally free in a neighbourhood of $C$
and the restriction $\sE_C$ is a nef vector bundle by Proposition \ref{prop:triv}  and \cite[Thm.1.2]{Fle84}. 
By construction of $P$ we have $\fibre{\pi}{C} \simeq \PP(\sE_C)$ and $\zeta|_{\fibre{\pi}{C}}=c_1(\sO_{\PP(\sE_C)}(1))$.
Hence $\zeta|_{\fibre{\pi}{C}}$ is a nef divisor.

If $\pi(Z)$ has codimension at least two in $X$, then the intersection $Z \cap \fibre{\pi}{C}$
is empty, and consequently the assertion of Lemma \ref{lemmacomputeclass} is trivially true. Thus we can assume from now on that
$$
\codim_X (\pi(Z)) \leq 1.
$$
Fix a very ample divisor $A$ on $P$. By \eqref{computeBminus} and 
\eqref{inclusionBminus} we find a $n_0 \in \N^*$ such that for all $n \geq n_0$
we have 
$$
Z \subset \mathbb B(\zeta+\frac{1}{n}A).
$$
We set $a_1 := \sigma_Z(\zeta)$ and observe that $a_1 > 0$ by \eqref{sigmapositive}. 
Since $\sigma_Z$ is a lower semicontinuous function we may suppose,  possibly enlarging $n_0$,
that
\begin{equation} \label{lowerbound}
\sigma_Z(\zeta+\frac{1}{n}A) \geq \frac{a_1}{2}
\end{equation}
for all $n \geq n_0$.
Since $$\mathbb B(\zeta+\frac{1}{n}A) \subset B_-(\zeta),$$ our hypothesis implies 
that if an irreducible component $W \subset \mathbb B(\zeta+\frac{1}{n}A)$ has codimension 
at most $k-1$, then $\pi(W)$ has codimension at least $2$ in $X$. 

In order to verify the inequality \eqref{basicineq}, 
let   $D_1, \ldots, D_k$ be very general effective $\Q$-divisors on $P$ such that 
$D_i \sim_\Q \zeta+\frac{1}{n}A.$ Then we have the following:
if 
$$
W \subset D_1 \cap \ldots \cap D_k
$$ 
is an irreducible component of codimension 
at most $k-1$, then $W$ is an irreducible component of $\mathbb B(\zeta+\frac{1}{n}A)$ and therefore
$\codim_X(\pi(W)) \geq 2.$ 
Thus for a very general curve $C \subset X$, the
intersection
$$
D_1 \cap \ldots \cap D_k \cap \fibre{\pi}{C}
$$
has pure codimension $k$ in $\fibre{\pi}{C}$. Since $C$ is general, the intersection
$Z \cap \fibre{\pi}{C}$ is reduced, thus by \eqref{lowerbound} 
$$
\mbox{mult}_{Z \cap \fibre{\pi}{C}} (D_j \cap \fibre{\pi}{C}) \geq \frac{a_1}{2},
$$
where $\mbox{mult}_{Z \cap \fibre{\pi}{C}} (\bullet)$ is the order of vanishing in any general point of $Z \cap \fibre{\pi}{C}$.
Consequently 
$$
\left(
D_1 \cap \ldots \cap D_k \cap \fibre{\pi}{C}
\right)
 - (\frac{a_1}{2})^k (Z \cap \fibre{\pi}{C})
$$ 
is an effective cycle of pure codimension $k$ in $\fibre{\pi}{C}$. Since
$$
[D_1 \cap \ldots \cap D_k \cap \fibre{\pi}{C}] = ((\zeta+\frac{1}{n}A)|_{\fibre{\pi}{C}})^k,
$$
we conclude 
$$
[
((\zeta+\frac{1}{n}A)|_{\fibre{\pi}{C}})^k - (\frac{a_1}{2})^k (Z \cap \fibre{\pi}{C})
]
\cdot H_1 \cdots H_{r-k} \geq 0
$$
for any collection of nef divisor $H_i$ on $\fibre{\pi}{C}$. Since $\frac{a_1}{2}$
does not depend on $n \geq n_0$ the statement now follows by setting $a:= (\frac{a_1}{2})^k$
and passing to the limit $n \to \infty$.
\end{proof}

\begin{remark}
In the proof above we can replace $\frac{a_1}{2}$ by $\frac{k-1}{k} a_1$ for an arbitrary
$k \in \N^*$. Taking the limit $k \to \infty$ the statement thus holds for $a=a_1^k$.
This is the natural bound following from more general considerations with
$(1,1)$-currents, see \cite[Cor.10.5]{Dem93}.
\end{remark}

\subsection{Proof of Theorem \ref{theoremmain}}
Let us start by showing that it is enough to prove that 
\begin{equation} \label{eqnc2}
c_2(\sE) \cdot H^{n-2}=0.
\end{equation}
In this case the existence of the quasi-\'etale cover such that $\nu^{[*]}(\sE)$ is numerically flat follows from \cite[Thm.1.20]{GKP16b}
(cf. also Remark \ref{remarksflat}).

If  additionally $X$ is smooth, then the cover  $\holom{\nu}{\tilde X}{X}$ is \'etale. 
Thus we see that $\nu^*(\sE) = \nu^{[*]} \sE$, which is locally free and numerically flat. 
Since $\mu$ is \'etale this implies that $\sE$ itself is locally free and numerically flat.

Let $\holom{\pi}{P = \PP_0(\sE)}{X}$ be a modification of $\PP(\sE) $ as in Definition \ref{definitiontautological},
and let $\zeta$ be a tautological class on $P$.
By Lemma  \ref{lemmapseff} our assumption implies that $\zeta$ is a pseudoeffective class.
Denote by $X_0 \subset X$ the locus where $X$ is smooth and $\sE$ is locally free.

For the proof of \eqref{eqnc2}, observe first that  
it is sufficient to show that 

{\em Claim.} {\em Let $W \subset B_-(\zeta)$ be an irreducible component.
Then $\codim_X (\pi(W)) \geq 2.$} 

Assuming this claim for the time being, let us see how to conclude:  consider a surface $S \subset  X_0$
cut out by very general elements of a sufficiently high multiple of $H$.
Since $X$ is smooth in codimension two, the surface $S$ is smooth and $\sE$ is locally free in a neighbourhood of $S$.
By construction of $P$ we have $\fibre{\pi}{S} = \PP(\sE_S)$ and $\zeta|_{\fibre{\pi}{S}}=c_1(\sO_{\PP(\sE_S)}(1))$.
Denote by $\zeta_S$ the restriction of $\zeta$ to $\PP(\sE_S)$. Then
$$
B_-(\zeta_S) \subset B_-(\zeta) \cap \PP(\sE_S),
$$
and therefore by the {\it Claim}, every irreducible component of $B_-(\zeta) \cap \PP(\sE|_S)$ is contained
in a fibre of the projection $\PP(\sE_S) \rightarrow S$. 
Thus the restriction of $\zeta_S$ to $B_-(\zeta_S)$ is nef (even ample), so
$\zeta_S$ is nef by \cite[Thm.2]{Pau98}. Since $c_1(\sE_S)=0$ this implies by \cite[Thm.2.5]{DPS94} that 
$$
c_2(\sE) \cdot S = c_2(\sE_S) = 0.
$$
Since the class of $S$ is a positive multiple of $H^{n-2}$ this proves Theorem \ref{theoremmain}. 

{\em Proof of the claim.}
We proceed by induction on $k \in \N^*$. The induction hypothesis
states that given any irreducible component $W \subset B_-(\zeta)$ of codimension at most $k-1$,  the image $\pi(W) \subset X$ has codimension at least $2$.

Note that for $k=1$ the unique subvariety of $P$ having codimension $k-1=0$
is $P$ itself. Since $\zeta$ is pseudoeffective by assumption, the total
space $P$ is not in $B_-(\zeta)$. Hence  the induction hypothesis holds for $k=1$.

Arguing by contradiction we suppose that there exists an irreducible component
$Z \subset B_-(\zeta)$ of codimension $k$ such that 
$$\codim_X(\pi(Z)) \leq 1.$$ 
Fix a number $l \in \N$ such that $\pi_* (\sI_Z(l))$ has positive rank and all
the higher direct images $R^i \pi_* (\sI_Z(l)) \vert X_0$ vanish. 

Let $C \subset X$ be a curve cut out by very general elements of the linear system $|m H|$, where we choose
$m \in \N$ sufficiently large so that the Mehta-Ramanathan theorem (in the form given by Lemma \ref{lemmaMR}) 
applies both for $\sE$ and $S^l \sE$, i.e.,
the restrictions $\sE_C$ and $S^l \sE_C$ are stable. Moreover,
$$ C \cap \pi(W) = \emptyset $$
where $W$ is any of the varieties appearing in
the induction hypothesis. 

Since $\sE_C$ is stable and $c_1(\sE_C)=0$, the restricted tautological class $\zeta_C$ is nef by Proposition \ref{prop:triv}. 
Thus we may apply
Lemma \ref{lemmacomputeclass} with $H_1=\ldots=H_{r-k}=\zeta_C$: there is a real number $a>0$ such that
$$
\zeta_C^r - a \zeta_C^{r-k} (Z \cap \fibre{\pi}{C}) \geq 0.
$$
Since $c_1(\sE_C)=0$, we have $\zeta_C^r=0$. Moreover, $\zeta_C$ being nef, we have
$$
\zeta_C^{r-k} (Z \cap \fibre{\pi}{C}) \geq 0.
$$ 
Consequently
$$
\zeta_C^{r-k} (Z \cap \fibre{\pi}{C}) = 0.
$$
Thus we see that $\zeta|_{Z \cap \pi^{-1}(C)}$ is not big.
Note that this already excludes the possibility that $\pi(Z)$ has codimension one in $X$:
in this case $Z \cap \fibre{\pi}{C}$ would be non-empty and contained in the $\pi$-fibres, so $\zeta_C$
would be ample on $ Z \cap \pi^{-1}(C).$ 

Thus we will assume from now on that $\pi(Z) = X.$ 
Since $S^l \sE_C$ is stable and since $\zeta|_{Z \cap \pi^{-1}(C)}$ is not big, Lemma \ref{lemmacurve} exhibits a 
curve $C' \subset (Z \cap \fibre{\pi}{C})$ such that 
$$
\zeta_C \cdot C'=0
$$
and $f:=\pi_{C'} : C' \rightarrow C$ is \'etale.
Clearly the vector bundle $f^* \sE_C$ is not stable since it has the numerically trivial quotient 
$f^* \sE_C \twoheadrightarrow \zeta|_{C'}$. 
On the other hand,  the vector bundle $\sE_C$ is stable and the holonomy group $H_x(\sE)=H_x(\sE_C)$
is connected by assumption, so the pull-back  $f^* \sE_C$ is stable by \cite[Lemma 6.22]{Dru17}.
Thus we have reached a contradiction. 
\begin{flushright} $\square$ \end{flushright}

\subsection{Proof of Corollary \ref{cor:leaves} }
By \cite[Prop.8.4]{Dru17}, it suffices to show that $\sF^*$ not pseudoeffective (cf. Definition \ref{definitionpseff}). 
Suppose to the contrary that $\sF^*$ is pseudoeffective. 
Since $X$ is simply connected, the algebraic holonomy group $H_x(\sF^*)$ is connected \cite[Prop.4]{BK08}. Hence we may apply Theorem \ref{theoremmain} and
conclude that $\sF^*$ is numerically flat. Thus we have $c_2(\sF) = c_2(\sF^*) = 0$ \cite{DPS94}, contrary to our assumption. 
\begin{flushright} $\square$ \end{flushright}

\section{Proof of Theorems \ref{theoremdecomp} and \ref{theorempseff} }

We start with the proof of the decomposition theorem.

\subsection {Proof of Theorem \ref{theoremdecomp}} 

We will follow the approach of Druel \cite{Dru17}.

By \cite[Thm.1.4]{Dru17} and \cite[Thm. B]{GGK17} there exists a quasi-\'etale finite cover $$\holom{\gamma}{A \times Z}{X}$$ where $A$ is an abelian variety and $Z$ a normal projective variety 
with the following properties. 
\begin{itemize} 
\item $Z$ has at most 
canonical singularities;
\item  $K_Z \simeq \sO_Z$ and  the augmented irregularity $\tilde q(Z)$ is zero;
\item there is a decomposition 
$$
T_Z \simeq \bigoplus_j \sE_j
$$
into reflexive integrable subsheaves $\sE_j$ of rank $m_j$ which are strongly stable in the sense of \cite[Defn.7.2]{GKP16c} for any polarization. 
\end{itemize} 

In order to simplify the notation we will suppose without loss of generality that $X=Z$. By Step 1 of the proof of 
\cite[Prop.4.10]{Dru17} we can suppose without loss of generality that $X$ has terminal $\Q$-factorial singularities.

By \cite[Thm. B and Prop.D]{GGK17} there exists a 
singular Ricci-flat K\"ahler metric, inducing a Riemannian metric $g$, such that for a smooth point $x \in X$, the decomposition
$$ 
T_{X,x} \simeq \bigoplus_j E_{j,x} 
$$ 
corresponds to the decomposition of $T_{X,x}$ 
into irreducible representations according to the action of the differential-geometric holonomy group $G$ of $g$ at $x$. 
Recall that $E_{j,x} := ( \sE_j)_x/m_x (\sE_j)_x$ with $m_x$ the maximal ideal in $x$. 
Moreover the differential-geometric holonomy groups $G_j$ of the direct factors $\sE_j$ are either $SU(m_j)$ or $Sp(\frac{m_j}{2})$, 
the representation $\rho_j: G_j \rightarrow \mbox{GL}(E_{j,x})$ being the standard one.   

By \cite[Prop.4.10]{Dru17}, it suffices to show that the leaves of the foliations $\sE_j$ are algebraic.  
By \cite[Prop. 8.4]{Dru17} it is furthermore sufficient to show that $\sE_j^*$ 
is not pseudoeffective (cf. Definition \ref{definitionpseff}). 

In order to simplify the notation we fix $j \in J$ and write $\sF = \sE_j^*$. We {\it claim} that, up to taking another
quasi-\'etale finite cover, the conditions (a) and (b) of Theorem \ref{theoremmain} are satisfied.
Once the claim is established, we argue by contradiction assume that $\sF$ is pseudoeffective. 
Then Theorem \ref{theoremmain} applies, and there exists a quasi-\'etale cover $g:\tilde X \to X$ 
such that $(g^* \sF)^{**}$ is a numerically flat vector bundle, contradicting \cite[Cor.5.11]{Dru17}.

{\em Proof of the claim.}
Since $\sF$ is strongly stable in the sense of \cite[Defn.7.2]{GKP16c} we may suppose by \cite[Lemma 40]{BK08}, \cite[Lemma 6.20]{Dru17}, 
possibly after passing to a quasi-\'etale cover, that the algebraic holonomy group $H_x(\sF)$ is connected.

Thus it remains to show that all
reflexive symmetric powers $S^{[l]}\sF$ are $H-$stable for some ample divisor $H$. By Lemma \ref{lemmaallstable}
it is sufficient to show that the algebraic holonomy group of $\sF$ is $\mbox{SL}(F_x)$ or $\mbox{Sp}(F_x)$.
We can now follow the proof of \cite[Thm.12.15]{GGK17}: 
by \cite[Prop.12.14]{GGK17}, it suffices to that $S^{[l]} \sF$ is indecomposable for {\it some} $l \geq 2$.
Now observe that the (differential-geometric) Bochner principle \cite[Thm. 8.1]{GGK17} also applies to the direct factors
of the tangent sheaf $T_X$. Thus any direct summand of $S^{[2]}\sF$ would create a $G$-invariant
subspace of the $G$-representation $S^2 F_x$. However, for $G=SU(F_x)$ and $G=Sp(F_x)$,
the induced representation on the second symmetric power is irreducible \cite{Wey49}, \cite[Sect.24.1 and 24.2]{FH91}.
Thus $S^{[2]} \sF$ is indecomposable and so are all the sheaves $S^{[m]} \sF$, completing the proof of Theorem \ref{theoremdecomp}. 
\begin{flushright} $\square$ \end{flushright}

\subsection{Proof of Theorem \ref{theorempseff} }

The claim is invariant under quasi-\'etale covers, so by \cite[Thm. E]{GGK17}, possibly after a
quasi-\'etale cover, the variety $X$ is Calabi-Yau or irreducible symplectic.
By \cite[Prop.12.14, Thm.12.15]{GGK17}
all the symmetric powers $S^{[l]}T_X$ are $H$-stable. 

Arguing by contradiction suppose that $T_X$ is pseudoeffective.  Hence by Theorem \ref{theoremmain}
we have $c_2(X) \cdot H^{n-2} = 0$. But then $X$ is a quasi-\'etale quotient of a torus by \cite[Thm.1.17]{GKP16b}, contradicting our assumption that $T_X$ is strongly stable. 
By duality all the symmetric powers $S^{[l]} \Omega_X^{[1]}$ are $H$-stable, so the proof for $\Omega_X^{[1]}$ is analogous. 
 \begin{flushright} $\square$ \end{flushright}
 
\section{Almost nef sheaves}

We start with some technical preparation. 

\begin{proposition} \label{propositionnumericallyflat}
Let $S$ be a smooth projective surface, 
and let $\sE$ be a locally free sheaf of rank $r$ over $S$ such that $c_1(\sE)=0$.
Suppose that $\sE$ is semistable with respect to some ample divisor $H$.
Suppose also that for some $a>0, b \in \N$ there exists an effective divisor
$D \in | a \zeta + b \pi^* H  |$ such that 
$$
\zeta^r \cdot D \geq 0.
$$
Then $\sE$ is numerically flat. 
\end{proposition}

\begin{proof}
By the classical Bogomolov-Gieseker inequality, \cite[4.7]{Miy87}, we have $c_2(\sE) \geq 0$. 
A theorem of Simpson \cite[Cor.3.10]{Sim92} states moreover that $c_2(\sE) = 0 $ if and only if $\sE$ is numerically flat. 
Thus it suffices to show that $c_2(\sE) \leq 0.$ 

By assumption we have
$$
\zeta^r \cdot D  \geq 0.
$$
On the other hand since $c_1(\sE)=0$ 
the Leray-Hirsch relation $$\zeta^r - \pi^* c_1(\sE) \zeta^{r-1} + \pi^* c_2(\sE) \zeta^{r-2} = 0$$ simplifies to
$\zeta^r = - \pi^* c_2(\sE) \zeta^{r-2}$. Consequently, 
$$
\zeta^r \cdot D = - \pi^* c_2(\sE) \zeta^{r-2} \cdot (a \zeta + b \pi^* H) = - a c_2(\sE) \geq 0.
$$
Since $a>0$ we arrive at $c_2(\sE) \leq 0$.
\end{proof}

\begin{lemma} \label{lemmaeliminatecurves}
Let $S$ be a smooth projective surface, and let $\sE$ be a locally free sheaf of rank $r \geq 2$ on $S$.
Let $(C_i)_{i \in I}$ be an at most countable collection of irreducible curves in $S$.

Let $H$ be an ample divisor on $S$. Then for every $m \gg 0$ there exists a divisor $D \in | \zeta + m (r-1) \pi^* H|$ with the following properties:
\begin{enumerate}
\item Let $\{p_1, \ldots, p_k\} \subset S$ be the non-flat locus of $\pi|_D : D \rightarrow S$.
Then $p_j \not\in C_i$ for every $j \in \{1, \ldots, k\}$ and $i \in I$  
\item For every $i \in I$ the restriction of the tautological class $\zeta|_{D \cap \fibre{\pi}{C_i}}$
is nef. 
\end{enumerate}
\end{lemma}

Recall first the following basic lemma:

\begin{lemma} \label{lemmaglobalsections}
Let $Z$ be a projective variety of dimension $d$, and let $\E$ be a locally free sheaf of rank $r$ on $Z$. Let $V \subset H^0(Z, \sE)$ be a linear subspace such that $\sE$ is generated by the
$V$, i.e. we have a surjective evaluation morphism
$$
\sO_Z \otimes V \rightarrow \sE.
$$
a) If $r>d$, then a general choice of $r-d$ elements of $V$ defines a subbundle
$$
\sO_Z^{\oplus r-d} \hookrightarrow \sE.
$$
b) If $r \geq d$, then a general choice of $r-d+1$ elements of $V$ defines an injective morphism
$$
\sO_Z^{\oplus r-d+1} \rightarrow \sE,
$$
that is a subbundle in the complement of finitely many points.
\end{lemma}

\begin{proof}
It is well-known \cite[II,Ex.8.2]{Har77} that if $r>d$, then a general section does not vanish, so a) follows by induction
on $r-d$. It is also well-known that if $r=d$, then a general section vanishes only in finitely many
points, so b) follows from a) and induction.
\end{proof}

\begin{proof}[Proof of Lemma \ref{lemmaeliminatecurves}]
For $m \gg 0$ we know by Serre's theorem that the sheaf $\sE^* \otimes \sO_S(mH)$ 
is globally generated; we denote by $V$ the space of global sections. 

For every $i \in I$, the restricted vector bundle $(\sE^* \otimes \sO_S(mH))|_{C_i}$ is generated by global sections of $V_i := V|_{C_i}$ (i.e. those global sections of $(\sE^* \otimes \sO_S(mH))|_{C_i}$ that lift
to global sections on $S$). By Lemma \ref{lemmaglobalsections} for a general choice of elements
$$
s_{1,i}, \ldots s_{r-1, i} \in V_i = V
$$
we obtain a subbundle
$$
\sO_{C_i}^{\oplus r-1} \rightarrow (E^* \otimes \sO_S(mH))|_{C_i}.
$$
Since there are only countably many curves $C_i$ we can thus fix very general sections
$$
s_{1}, \ldots, s_{r-1} \in V
$$
inducing an injective morphism
$$
s_{1} \oplus \ldots \oplus s_{r-1}: \sO_S^{\oplus r-1} \rightarrow \sE^* \otimes \sO_S(mH),
$$
such that the 
the restriction
$$
s_{i,1} \oplus \ldots \oplus s_{i,r-1}:  \sO_{C_i}^{\oplus r-1} \rightarrow (\sE^* \otimes \sO_S(mH))|_{C_i}.
$$
defines a subbundle.

Moreover by part b) of Lemma \ref{lemmaglobalsections}
the map is injective in the complement of finitely many points $p_1, \ldots, p_k$.
Dualising and tensoring with $\sO_S(mH)$ we obtain a morphism
$$
\sE \rightarrow \sO_S(mH)^{\oplus r-1}
$$ 
that is surjective in the complement of finitely many points and the restriction
$$
\sE|_{C_i} \rightarrow \sO_{C_i}(mH)^{\oplus r-1}
$$
is surjective. In particular the restriction of $\zeta$ to the divisor $\PP(\sO_{C_i}(mH)^{\oplus r-1}) \subset \PP(\sE|_{C_i})$ is nef (even ample). The map $\sE \rightarrow \sO_S(mH)^{\oplus r-1}$ is in general not surjective (so does not define a subvariety of $\PP(\sE)$) however if we denote by $\sH$ its image, then
$\PP(\sH) \subset \PP(\sE)$ is a prime divisor $D$ which is easily seen to be in the
linear system  $| \zeta + m (r-1) \pi^* H |$. 
\end{proof}

\begin{proof}[Proof of Theorem \ref{theorem:almostnef}] Let $H$ be a very ample line bundle on $X$.

{\em Step 1. The case $\dim X=2$.} Then by assumption $X$ is smooth, hence the reflexive sheaf $\sE$ is locally free.
By Proposition \ref{propositionnumericallyflat} it is sufficient to find a divisor $D \in |\zeta + l \pi^* H|$ 
such that $\zeta^r \cdot D \geq 0$.

Since $\sE$ is almost nef and $\dim X = 2$, there exists an at most countable collection of curves $C_i \subset X$
such that $\sE_{C_i}$ is not nef. Consequently if $C \subset \PP(\sE)$ is a curve
such that $\zeta \cdot C < 0$ then $C$ maps onto one of the curves $C_i$.
By Lemma \ref{lemmaeliminatecurves} we can find a divisor $D \in |\zeta + \pi^* (r-1)m H|$
such that  $\zeta|_{D \cap \fibre{\pi}{C_i}}$ is nef for all $i$. Thus $\zeta_D$ is nef,
in particular 
$$
\zeta^r \cdot D = (\zeta|_D)^r \geq 0.
$$

{\em Step 2. Reduction to the case $\dim X=2$.} 
Let $S \subset X$ be a surface cut out by very general hyperplane sections $D_i \in | H |$. Then the restriction $\sE_S$ to $S$
is almost nef and $\det \sE_S$ is numerically trivial. Since $X$ is smooth in codimension two, the surface $S$ is smooth.
Since $\sE$ is reflexive, the restriction $\sE_S$ is locally free. By Step 1 we know that $\sE_S$ is numerically flat,
so $c_2(\sE_S)=0$ by \cite[Thm.2.5]{DPS94}. Since the class of $S$ is a positive multiple of $H^{\dim X-2}$, we obtain
that $c_2(\sE) \cdot H^{\dim X-2}=0$.
Now \cite[Thm.1.20]{GKP16b} (cf. also Remark \ref{remarksflat}) provides a quasi-\'etale cover $\gamma: \tilde X \to X$ such that $\tilde \sE := \gamma^*(\sE)^{**}$ is locally free and numerically flat.
\end{proof}

Using that almost nefness and numerically flatness of vector bundles 
are invariant under a birational morphism, the following variant of Theorem \ref{theorem:almostnef} follows by passing
to a desingularisation.

\begin{corollary} 
Let $X$ be a normal projective variety, and let $\sE$ be a locally free sheaf on $X$ such that $c_1(\sE) = 0.$ 
If $\sE$ is almost nef, then  $\sE$ is numerically flat. 
\end{corollary} 

\begin{proof}[Proof of Theorem \ref{corollaryalmostnef}] 
Assume that $T_X$ is almost nef. 
By Theorem \ref{theorem:almostnef}, there exists a quasi-\'etale cover $\gamma: \tilde X \to X$ such that 
$\gamma^{[*]}(T_X) $ is locally free and numerically flat. Since $T_{\tilde X} =  \gamma^{[*]}(T_X), $ we conclude by \cite[Thm. 6.1]{GKKP11},  that $\tilde X$ is smooth and
$$ c_2(\tilde X) = 0.$$ 
Hence by Yau's theorem $\tilde X$ is an \'etale quotient of a torus. 
\end{proof}

\newcommand{\etalchar}[1]{$^{#1}$}

\end{document}